\documentclass[12pt]{amsart}
\usepackage{amscd,amsmath,amsthm,amssymb}
\usepackage{pstcol,pst-plot,pst-3d}
\usepackage{lmodern,pst-node}
\def\NZQ{\Bbb}               
\def\NN{{\NZQ N}}

\def\ZZ{{\NZQ Z}}

\def\frk{\frak}               

\def\mm{{\frk m}}

\def\Phi{{\frk n}}
\def\Phi{{\frk N}}
\def\opn#1#2{\def#1{\operatorname{#2}}} 
\opn\chara{char} \opn\length{\ell} \opn\pd{pd} \opn\rk{rk}
\opn\projdim{proj\,dim} \opn\injdim{inj\,dim} \opn\rank{rank}
\opn\depth{depth} \opn\grade{grade} \opn\height{height}
\opn\embdim{emb\,dim} \opn\codim{codim}

\opn\Tr{Tr} \opn\bigrank{big\,rank}
\opn\superheight{superheight}\opn\lcm{lcm}
\opn\trdeg{tr\,deg}
\opn\reg{reg} \opn\lreg{lreg} \opn\ini{in} \opn\lpd{lpd}
\opn\size{size}\opn\bigsize{bigsize}
\opn\cosize{cosize}\opn\bigcosize{bigcosize}
\opn\sdepth{sdepth}\opn\sreg{sreg}
\opn\link{link}\opn\fdepth{fdepth}
\opn\div{div} \opn\Div{Div} \opn\cl{cl} \opn\Cl{Cl}
\opn\Spec{Spec} \opn\Supp{Supp} \opn\supp{supp} \opn\Sing{Sing}
\opn\Ass{Ass} \opn\Min{Min}\opn\Mon{Mon} \opn\dstab{dstab} \opn\astab{astab}
\opn\Ann{Ann} \opn\Rad{Rad} \opn\Soc{Soc}
\opn\Im{Im} \opn\Ker{Ker} \opn\Coker{Coker} \opn\Am{Am}
\opn\Hom{Hom} \opn\Tor{Tor} \opn\Ext{Ext} \opn\End{End}
\opn\Aut{Aut} \opn\id{id}

\opn\nat{nat}
\opn\pff{pf}
\opn\Pf{Pf} \opn\GL{GL} \opn\SL{SL} \opn\mod{mod} \opn\ord{ord}
\opn\Gin{Gin} \opn\Hilb{Hilb}\opn\sort{sort}
\opn\aff{aff} \opn\con{conv} \opn\relint{relint} \opn\st{st}
\opn\lk{lk} \opn\cn{cn} \opn\core{core} \opn\vol{vol}
\opn\link{link} \opn\star{star}\opn\lex{lex}
\opn\gr{gr}

\def\pot#1#2{#1[\kern-0.28ex[#2]\kern-0.28ex]}

\opn\dirlim{\underrightarrow{\lim}}
\opn\inivlim{\underleftarrow{\lim}}
\let\sect=\cap
\let\dirsum=\oplus

\let\iso=\cong
\let\Union=\bigcup

\let\Dirsum=\bigoplus

\let\to=\rightarrow
\let\To=\longrightarrow
\def\Implies{\ifmmode\Longrightarrow \else
        \unskip${}\Longrightarrow{}$\ignorespaces\fi}
\def\implies{\ifmmode\Rightarrow \else
        \unskip${}\Rightarrow{}$\ignorespaces\fi}
\def\iff{\ifmmode\Longleftrightarrow \else
        \unskip${}\Longleftrightarrow{}$\ignorespaces\fi}

\let\:=\colon
\newtheorem{Theorem}{Theorem}[section]
\newtheorem{Lemma}[Theorem]{Lemma}
\newtheorem{Corollary}[Theorem]{Corollary}
\newtheorem{Proposition}[Theorem]{Proposition}

\newtheorem{Example}[Theorem]{Example}

\let\epsilon\varepsilon
\let\kappa=\varkappa
\def\qed{\ifhmode\textqed\fi
      \ifmmode\ifinner\quad\qedsymbol\else\dispqed\fi\fi}
\def\textqed{\unskip\nobreak\penalty50
       \hskip2em\hbox{}\nobreak\hfil\qedsymbol
       \parfillskip=0pt \finalhyphendemerits=0}
\def\dispqed{\rlap{\qquad\qedsymbol}}

\opn\dis{dis}
\def\pnt{{\raise0.5mm\hbox{\large\bf.}}}

\opn\Lex{Lex}

\begin{document}
\title {Squarefree monomial ideals with constant depth function  }
\author {J\"urgen Herzog, Marius Vladoiu }

\address{J\"urgen Herzog, Fachbereich Mathematik, Universit\"at Duisburg-Essen, Campus Essen, 45117 Essen, Germany} \email{juergen.herzog@uni-essen.de}

\address{Marius Vladoiu, Faculty of Mathematics and Computer Science, University of Bucharest, Str. Academiei 14, Bucharest, RO-010014, Romania}\email{vladoiu@gta.math.unibuc.ro}

\thanks{This paper was written during the visit of the second author at the Universit\"at Duisburg-Essen, Campus Essen. The second author was supported by a Romanian grant awarded by UEFISCDI, project number $83/2010$, PNII-RU code TE$\_46/2010$, program Human Resources, ``Algebraic modeling of some combinatorial objects and computational applications''.}

\subjclass{13C13, 13A30, 13F99,  05E40}
\keywords{edge ideals, matroidal ideals, simplicial forests, depth functions, analytic spread}
\begin{abstract}
In this paper we study squarefree monomial ideals which have constant depth functions. Edge ideals, matroidal ideals and facet ideals of pure simplicial forests connected in codimension one with this property  are classified.
\end{abstract}
\maketitle

\section*{Introduction}

\bigskip

For a graded ideal $I$ in a polynomial ring $S=K[x_1,\ldots,x_n]$ over a field $K$ the depth function of $I$ is defined to be the numerical function $f:\NN\To\NN$, $k\mapsto \depth S/I^k$. This depth function has been studied by several authors in \cite{HH1},\cite{BHH},\cite{HRV}. One of the main problems in this context is to characterize those numerical functions which are the depth functions of a graded ideal. The answer to this problem is widely open. On the other hand by a classical result of Brodmann \cite{Br1} any depth function is eventually constant. In other words, for any graded ideal $I\subset S$ there exists an integer $t_0$ such that $\depth S/I^t$ is constant for all $t\geq t_0$. We call this constant depth by limit depth and denoted by $\lim_{t\to \infty}\depth S/I^t$. Brodmann's theorem is actually valid for any ideal in a Noetherian local ring. However in this paper we restrict ourselves to the case of monomial ideals. Though the depth function is not well understood in general, it has been shown in \cite{HH1} that any bounded non--decreasing numerical function is the depth function of a suitable monomial ideal. In contrast to this result it is expected by several authors \cite{FHV} that the depth function of a squarefree monomial ideal is a non--increasing numerical function. For this statement it is important to require that the ideal $I$ is indeed squarefree, because it has been recently shown \cite{BHH} that if the monomial $I$ is not squarefree, then its depth function may have any number of local maxima.

In this paper, we aim at classifying those monomial ideals whose depth function is constant. Such ideals will be called ideals with constant depth functions. By Brodmann's theorem \cite{Br1} any high enough power of an ideal has a constant depth function. Therefore, we restrict our classification problem to squarefree monomial ideals. In this case $\depth S/I\geq \depth S/I^k$ simply because $\depth \sqrt{J}\geq\depth J$ for any monomial ideal $J$, see for example \cite[Theorem 2.6]{HTT}. In the same paper \cite{HTT} and also in \cite{AE} it is studied the question, related to our problem, when certain classes of monomial ideals with a given radical have the same depth. Considering the powers of an ideal $I$, as done in this paper, there is a classical result by Waldi \cite[Korollar 1]{W}, which asserts that if $I$ is generically complete intersection and all powers of $I$ have maximal depth, that is,  they are Cohen--Macaulay, then $I$ is a complete intersection. On the other hand, the class of squarefree monomial ideals with constant depth functions, whose powers are not Cohen--Macaulay, is much larger.

In the first section of this paper we describe a method of constructing squarefree monomial ideals with constant depth function. In Theorem~\ref{constsumprod} it is shown that if $I$ and $J$ are squarefree monomial ideals in disjoint sets of variables whose Rees rings are Cohen--Macaulay then  $I+J$ has a  constant depth function if and only if $I$ and $J$ have this property. A similar statement holds for $IJ$. The proof of this result is less obvious than one might expect. Even in the simple case that  $f$ is a non--zero divisor modulo $I$ it is not clear to us how the depth function of the ideal $I$ is related to that of $(I,f)$. The proof of Theorem~\ref{constsumprod} relies on the following fact, presented  in Corollary~\ref{constlim}, where it is stated that a monomial ideal $I$, whose Rees ring is Cohen--Macaulay has a constant depth function if and only if $\depth S/I=n-\ell(I)$,  where $\ell(I)$ denotes the analytic spread of $I$. This criterion is an immediate consequence of a result \cite[Proposition 3.3]{EH} of Eisenbud and Huneke.

By an iterated application of Theorem~\ref{constsumprod} one obtains as a special case the following class $\mathcal{C}$ of ideals with constant depth function: $I\in\mathcal{C}$ if and only if $I=I_1+\cdots +I_k$, where the ideals $I_j$ are defined in pairwise disjoint sets of variables and where each $I_j$ itself is a product of monomial prime ideals in pairwise disjoint sets of variables. Unfortunately, as shown in Example~\ref{bad} not all squarefree monomial ideals with constant depth functions are of the form described in Example~\ref{good}. The more it is surprising that any edge ideal (Theorem~\ref{constantdepth}), any matroidal ideal (Theorem~\ref{criterionmatroidconstdepth}), as well as any facet ideal of a pure simplicial forest connected in codimension one (Theorem~\ref{consttree}) with constant depth function belongs to  the class $\mathcal{C}$. This is the content of Section 2.

\bigskip

\section{Construction of squarefree monomial ideals with constant depth function}

\bigskip

Throughout this paper we denote by $S=K[x_1,\ldots,x_n]$ the polynomial ring in $n$ variables over the field $K$, and by $\mm$ the graded maximal ideal of $S$. Moreover if $I$ is a monomial ideal of $S$ we denote as usual by $G(I)$ the unique set of minimal monomial generators of $I$. The main purpose of this section is to prove the following

\begin{Theorem}
\label{constsumprod}
Let $I,J$ be monomial ideals of $S$ generated  in disjoint sets of variables with the property that $\mathcal{R}(I)$ and $\mathcal{R}(J)$ are Cohen--Macaulay. Then $\mathcal{R}(I+J)$ and $\mathcal{R}(IJ)$ are Cohen--Macaulay. Moreover, the following conditions are equivalent:
\begin{enumerate}
\item[(i)] the depth functions of $I$ and  $J$ are constant;
\item[(ii)] the depth function of $I+J$ is constant;
\item[(iii)] the depth function of $IJ$ is constant.
\end{enumerate}
\end{Theorem}

Starting with monomial prime ideals and applying Theorem~\ref{constsumprod} iteratively one obtains the following family of  monomial ideals,  described in the next corollary, whose depth function is constant.

To describe this family we first define the set $\mathcal{A}$ whose elements are collections $A=\{A_1,\ldots,A_r\}$ of subsets of $[s]$ (including the empty collection) satisfying:
\begin{enumerate}
\item[(i)] $A\in\mathcal{A}$ if $|A_i|=1$ for $i=1,\ldots,r$.
\item[(ii)] For each $j\in [s]$ set $A(j)=\{i\in [r]\:\; j\in A_i\}$. Then there exists $j\in [s]$ such that
\[
\Union_{i\in A(j)}(A_i\setminus\{j\})\sect\Union_{i\not\in A(j)}A_i=\emptyset,
\]
and the collections $\{A_i\setminus\{j\}\:\; i\in A(j)\}$ and $\{A_i\:\; i\not\in A(j)\}$ belong again to $\mathcal{A}$.
\end{enumerate}

\begin{Corollary}
\label{hope}
Let $P_1,\ldots, P_s$ be monomial prime ideals in pairwise disjoint sets of variables, and let $\{A_1,\ldots, A_r\}$ be a collection of subsets of $[s]$ belonging to the set $\mathcal{A}$,  described before. Then the monomial ideal
\[
I=I_1+I_2+\cdots +I_r \quad \text{with} \quad I_j=\prod_{i\in A_j}P_i \quad\text{for}\quad j=1,\ldots,r
\]
has a constant depth function.
\end{Corollary}

The following examples demonstrate this construction.

\begin{Example}\label{good}
{\em (i) Let $P_1,\ldots,P_8\subset S$ be monomial prime ideals in pairwise disjoint sets of variables, and let $I$ be the following ideal of $S$
\[
I=P_1P_2P_5P_8+P_1P_3P_5P_8+P_4P_5P_8+P_6P_7P_8.
\]
The depth function of $I$ is constant since it is an ideal as described in Corollary~\ref{hope}, as can be seen from the following presentation
\[
I=P_8(P_5(P_1(P_2+P_3)+P_4)+P_6P_7).
\]

(ii) We denote by $\mathcal{C}$ the family of those monomial ideals which are defined as in Corollary~\ref{hope} by  collections  $\{A_1,\ldots,A_r\}$  of subsets of $[s]$ with $A_i\sect A_j=\emptyset$ for $i\neq j$. Since such collections obviously belong to $\mathcal{A}$, it follows that all monomial ideals in $\mathcal{C}$ have a constant depth function.
}
\end{Example}

 It will be shown in the next section that any squarefree monomial ideal generated in degree $2$ belongs to the family $\mathcal{C}$  described in Example~\ref{good}. However, this is no longer the case for squarefree monomial ideals generated in degree $3$. The following example does not even belong to the larger class of monomial ideals described in Corollary~\ref{hope}.

\begin{Example}\label{bad}
{\em The ideal $I=(x_1x_2x_3, x_3x_4x_5, x_1x_5x_6)\subset S=K[x_1,\ldots,x_6]$ has the property that $\depth S/I^k=3$ for all $k$, and does not belong to any of the families of monomial ideals described before. It will be explained after Corollary~\ref{constlim} why the depth function of this ideal is constant.}
\end{Example}

In order to prove Theorem~\ref{constsumprod} we need some preparations and to recall some basic facts. One of this facts is the theorem of Burch \cite{Burch}
\[
\lim_{t\to \infty}\depth S/I^t\leq n-\ell(I),
\]
where $\ell(I)$ is the {\em analytic spread} of $I$. In other words $\ell(I)$ is the Krull dimension of the {\em fiber ring} $\mathcal{\overline{R}}(I)=\mathcal{R}(I)/\mm\mathcal{R}(I)$ of the Rees ring $\mathcal{R}(I)$ of $I$. We will also use the result, due to Eisenbud and Huneke \cite[Proposition 3.3]{EH}, which says that equality holds in the Burch inequality  if the associated graded ring of $I$ is Cohen--Macaulay, which for example is the case if   $\mathcal{R}(I)$ is Cohen--Macaulay, see \cite{Hu}. Since we want to apply these results we have to understand how the analytic spread behaves with respect to sum and product of monomial ideals in disjoint sets of variables. In the special case of two monomial ideals each of them generated in a single degree, the following proposition regarding the sum was observed in \cite[Lemma 3.4]{BMV}.

\begin{Proposition}
\label{addanalytic}
Let $I,J$ be monomial ideals of $S$ generated in disjoint sets of variables. Then  $\ell(I+J)=\ell(I)+\ell(J)$ and $\ell(IJ)=\ell(I)+\ell(J)-1$.
\end{Proposition}
\begin{proof}
We denote by $H_1,H_2$ the Hilbert functions of $\mathcal{\overline{R}}(I)$, respectively $\mathcal{\overline{R}}(J)$, that is,  $H_1(k)=\dim_{K}I^k/\mm I^k$ for all $i$, and similarly for $H_2(K)$. By Hilbert's theorem \cite[Theorem 6.1.3]{HH} we have
\[
\sum_{k\geq 0} H_1(k)t^k = \frac{Q_1(t)}{(1-t)^{\ell(I)}}\quad \text{and}\quad H_2(k)t^k = \frac{Q_1(t)}{(1-t)^{\ell(J)}}
\]
for $i=1,2$, where $Q_1,Q_2\in\ZZ[t]$ with $Q_1(1)>0$ and $Q_2(1)>0$. In order to prove the sum formula we notice  that $$\sum_{k\geq 0}(\sum_{i=0}^k H_1(i)H_2(k-i))t^k=\frac{Q_1(t)Q_2(t)}{(1-t)^{\ell(I)+\ell(J)}}.$$
Since $Q_1(1)Q_2(1)>0$ we obtain that $\ell(I+J)=\ell(I)+\ell(J)$, provided that the Hilbert function $H$ of $\mathcal{\overline{R}}(I+J)$ satisfies $H(k)=\sum_{i=0}^k H_1(i)H_2(k-i)$. The latter statement is equivalent to proving
\[
G(I^iJ^{k-i})=G(I^i)G(J^{k-i})\subset G((I+J)^k)
\]
for each $i$, with $0\leq i\leq k$, where as usual $G(L)$ denotes the unique minimal monomial set of generators of a monomial ideal $L$. The equality is an immediate consequence of the fact that $I,J$ are monomial ideals in disjoint sets of variables. In order to prove the above inclusion we argue by contradiction. Suppose that there exists a monomial $w\in G(I^iJ^{k-i})\setminus G((I+J)^k)$. Then there exists an integer $j$ with $i\neq j$ such that $w\in I^j J^{k-j}$. Therefore there exist $u_l\in G(I^l), v_l\in G(J^{k-l})$ for $l=i,j$ such that $w=u_iv_i$ and $w$ is divisible by $u_jv_j$. On the other hand, since $I,J$ are monomial ideals in disjoint sets of variables, it follows that $u_jv_j$ divides $u_iv_i$ if and only if $u_j$ divides $u_i$ and $v_j$ divides $v_i$. These two relations of divisibility cannot hold simultaneously since $i\neq j$. Indeed, if $i<j$ then $u_j$ does not divide $u_i$ and if $i>j$ then $v_j$ does not divide $v_i$. Hence, we obtain a contradiction to our assumption and we are done.

For the statement concerning $IJ$, let us denote by $H'$ the Hilbert function of $\mathcal{\overline{R}}(IJ)$. One can easily see that $H'(k)=H_1(k)H_2(k)$ for all $k$, which implies that the same equality holds for the corresponding Hilbert polynomials. Passing to the degrees one obtains the desired equality.
\end{proof}

We recall the following result \cite[Proposition 3.3]{EH} of Eisenbud and Huneke which,  for the convenience of the reader,  we restate it in the frame and terminology introduced so far.

\begin{Proposition}
\label{normal}
Let $I$ be a monomial ideal of $S$ such that $\mathcal{R}(I)$ is Cohen--Macaulay. Then $\min\{\depth S/I^t \: t\geq 1\}=n-\ell(I)$. Moreover,  if  $l$ is the minimum integer such that $\depth S/I^l=\min\{\depth S/I^t \: t\geq 1\}$,  then  $\depth S/I^t=\depth S/I^l$ for all $t\geq l$.  In particular,   $\lim_{t\to \infty}\depth S/I^t=\min\{\depth S/I^t \: t\geq 1\}=n-\ell(I)$.
\end{Proposition}

The following corollary will be crucial for the further considerations.

\begin{Corollary}
\label{constlim}
Let $I\subset S$ be a monomial ideal  such that $\mathcal{R}(I)$ is Cohen--Macaulay. Then the depth function of $I$ is constant if and only if $\depth S/I=n-\ell(I)$.
\end{Corollary}
\begin{proof}
It follows from Proposition~\ref{normal} that if $\depth S/I=n-\ell(I)$, then  $\depth S/I=\min\{\depth S/I^t \: t\geq 1\}$. Therefore, the minimum integer $l$ such that $\depth S/I^l=\min\{\depth S/I^t \: t\geq 1\}$ is one.  Applying again Proposition~\ref{normal} we obtain that  $\depth S/I=\depth S/I^t$ for all $t\geq 1$. Hence the depth function of $I$ is constant. The other implication is obvious, due to the fact that $\min\{\depth S/I^t \: t\geq 1\}=n-\ell(I)$.
\end{proof}

Coming back to Example~\ref{bad} it can be easily checked, for example by using  CoCoA~\cite{Co},  that  $n-\ell(I)=6- 3=\depth S/I$ and that the Rees ring of $I$ is Cohen--Macaulay. Thus the preceding corollary implies that $I$ has a constant depth function.

\medskip
\begin{proof}[Proof of Theorem~\ref{constsumprod}]
There are well--known facts that $\mathcal{R}(I+J)$ and $\mathcal{R}(IJ)$ are Cohen--Macaulay (see \cite[Theorem 4.7]{SVV}, respectively \cite[Corollary 2.10]{Hy}). We will prove the equivalent statements of the theorem by showing $(i)\Rightarrow (ii)\Rightarrow (iii)\Rightarrow (i)$.

(i)$\Rightarrow$(ii):  Since $I$ and $J$ are ideals in disjoint sets of variables, we may assume that there exist monomial ideals $I_0\subset S_1=K[x_1,\ldots,x_m]$ and $J_0\subset S_2=K[x_{m+1},\ldots,x_n]$ for some integer $m$ with $1\leq m< n$ such that $I=I_0S$ and $J=J_0S$. Then it follows  from \cite[Theorem 2.2.21]{Vi} that $\depth_S (S/(I+J))=\depth_{S_1} (S_1/I_0) + \depth_{S_2} (S_2/J_0)$. In addition, $\mathcal{R}(I_0)$ and $\mathcal{R}(J_0)$ are Cohen--Macaulay since $\mathcal{R}(I)$ and $\mathcal{R}(J)$ are Cohen--Macaulay. Therefore, by Corollary~\ref{constlim}, $\depth_{S_1} (S_1/I_0)=m-\ell(I_0)$ and $\depth_{S_2} (S_2/J_0)=n-m-\ell(J_0)$. Hence,
\[
\depth S/(I+J)=n-\ell(I)-\ell(J)=n-\ell(I+J),
\]
since $\ell(I)=\ell(I_0)$ and $\ell(J)=\ell(J_0)$ and since by  Proposition~\ref{addanalytic}, $\ell(I+J)=\ell(I)+\ell(J)$. Thus Corollary~\ref{constlim} implies that the depth function of $I+J$ is constant.

(ii)$\Rightarrow$(iii): In order to prove that the depth function of $IJ$ is constant we consider the following short exact sequence
\begin{equation}
\label{key}
0\To S/(I\cap J)\To S/I \dirsum S/J \To S/(I+J)\To 0.
\end{equation}
Since the depth function of $I+J$ is constant and $I,J$ are monomial ideals in disjoint sets of variables we have $\depth S/(I+J)=n-\ell(I)-\ell(J)$ and $I\cap J=IJ$.

We  first observe  that $$\lim_{t\to \infty}\depth S/(IJ)^t=n-\ell(IJ)=n+1-\ell(I)-\ell(J),$$
 since $\mathcal{R}(IJ)$ is Cohen-Macaulay.  The last equality follows from Proposition~\ref{addanalytic}. Considering  the exact sequence (\ref{key}) we find, by applying the Depth Lemma  (see for example \cite[Proposition 1.2.9]{BH}),  that $\depth S/IJ=\depth S/(I+J)+1$, and hence $\depth S/IJ=n+1-\ell(I)-\ell(J)$. Therefore, Corollary~\ref{constlim} we implies  that the depth function of $IJ$ is constant.

(iii)$\Rightarrow$(i): Observe that
\begin{equation}\label{one}
\depth_{S_1}(S_1/I_0)\geq \lim_{t\to \infty}\depth_{S_1}(S_1/{I_0^t})=m-\ell(I_0)=m-\ell(I)
\end{equation}
and
\begin{equation}\label{two}
\depth_{S_2}(S_2/J_0)\geq \lim_{t\to \infty}\depth_{S_2}(S_2/{J_0^t})=n-m-\ell(J_0)=n-m-\ell(J).
\end{equation}
For these inequalities we used the fact that for any squarefree monomial ideal $L$ one has $\depth L\geq \depth L^t$ for all $t$, as noted in the introduction.

Since $I$ and $J$ are monomial ideals in disjoint sets of variables we have
\begin{equation}\label{5}
\depth S/IJ=\depth S/(I+J)+1=\depth_{S_1}(S_1/I_0) + \depth_{S_2}(S_2/J_0)+1.
\end{equation}
It follows from (\ref{one}), (\ref{two}) and (\ref{5}) that
\[
\depth S/IJ\geq n+1-\ell(I)-\ell(J),
\]
where equality holds if and only if equality holds in (\ref{one}) and (\ref{two}). On the other hand, by Corollary~\ref{constlim}
\begin{equation}\label{zero}
\depth S/IJ=n+1-\ell(I)-\ell(J),
\end{equation}
since we assume that $IJ$ has a constant depth function. Therefore we have equality in (\ref{one}) and (\ref{two}) which by  Corollary~\ref{constlim} implies that the depth functions of $I_0$ and $J_0$ are constant. Consequently we also have that the depth functions of $I$ and $J$ are constant.
\end{proof}

\bigskip

\section{Classes of squarefree monomial ideals with constant depth functions}

\bigskip

 The purpose of this section is to prove that the edge ideals, matroidal ideals and facet ideals of pure simplicial forests connected in codimension one whose depth functions are constant belong to the class $\mathcal{C}$ defined in the introduction.

To begin with, let $G$ be a finite simple graph, $V(G)=[n]$ its set of vertices and $E(G)$ its set of edges. We denote as usual by $I(G)$ the {\em edge ideal} of the graph $G$, which is defined to be the ideal of $S=K[x_1,\ldots,x_n]$, generated by the monomials $x_ix_j$ such that $\{i,j\}\in E(G)$. Identifying each vertex $i$ with the variable $x_i$, we have $S=K[V(G)]$. Let $X$ be a subset of $V(G)$. Then the graph $G\setminus X$ is the graph on the vertex set $V(G)\setminus X$ with the set of edges $$E(G\setminus X)=\{\{x_i,x_j\}\in E(G): x_i,x_j\in V(G\setminus X) \}.$$

Let $X$ be the set of isolated vertices of $G$. Then
\[
I(G)=I(G\setminus X)K[V(G)].
\]
From this follows that the depth function of $I(G)$ is constant if and only if the depth function of $I(G\setminus X)$ is constant. Therefore, unless otherwise stated, we will always assume that $G$ has no isolated vertices.

\begin{Proposition}
\label{bipdepth1}
Let $G$ be a bipartite graph and $I(G)$ its edge ideal. Then $G$ is a complete bipartite graph if and only if $\depth S/{I(G)}=1$.
\end{Proposition}

\begin{proof}
Let $V(G)=\{x_1,\ldots,x_m\}\cup\{y_1,\ldots,y_n\}$ be the bipartition of the vertex set of $G$. If $G$ is a complete bipartite graph then $I(G)=(x_1,\ldots,x_m)(y_1,\ldots,y_n)$, and  we obtain that $\depth S/I(G)=1$, see for example \cite[Theorem 3.14]{HRV}. For the converse, let us notice first that if $G$ is disconnected then $\depth S/I(G)\geq 2$. Indeed, if $G_1,\ldots,G_k$ with $k\geq 2$ are the connected components of $G$, then
\[
\depth S/I(G)=\depth K[V(G_1)]/I(G_1)+\cdots+\depth K[V(G_k)]/I(G_k)\geq k,
\]
where the equality follows from \cite[Theorem 2.2.21]{Vi} while the inequality follows from the fact that $\depth K[V(G_i)]/I(G_i)\geq 1$ for all $i$. Hence $\depth S/I(G)=1$ implies that $G$ is connected.

We prove by induction on $n+m$, the number of vertices of $G$, that a bipartite connected graph $G$ which is not complete has $\depth S/I(G)\geq 2$. The first such case of a  graph $G$ is when $n=m=2$ and $|E(G)|=3$. We may assume that $E(G)=\{\{x_1,y_1\},\{x_1,y_2\},\{x_2,y_2\}\}$ in which case one can easily check that $\depth S/I(G)=2$. For the induction step, let $m+n\geq 5$. Since $G$ is connected and  not complete we have $m,n\geq 2$. This  implies that at least one integer, say $m$, is greater than or equal to three. In addition, the fact that $G$ is not complete implies that there exist integers $i,j$ with $1\leq i\leq m$ and $1\leq j\leq n$ such that $\{x_i,y_j\}$ is not an edge of $G$. Let  $l\neq i$ be an integer with $1\leq l\leq m$,  and consider the following short exact sequence
\[
0\To  S/(I(G):(x_l)) \stackrel{x_l}{\To} S/I(G) \To S/(I(G),x_l) \To 0.
\]
We have the following ring isomorphisms
\[
S/(I(G):(x_l))\iso K[V(G')][x_l]/I(G') \quad \text{ and } \quad S/(I(G),x_l)\iso K[V(G'')]/I(G''),
\]
where $G'$ is the graph $G\setminus (N_{x_l}(G)\cup \{x_l\})$ and $G''$ is the graph $G\setminus\{x_l\}$. We recall that by $N_{x_l}(G)$ we denote, as usual, the set of neighbors of $x_l$ in the graph $G$, that is, the set of all vertices $x_p$ of $G$ such that $\{x_l,x_p\}\in E(G)$. Since $m\geq 3$, the graph $G'$ has at least two vertices and consequently $\depth K[V(G')]/I(G')\geq 1$. Therefore $\depth S/(I(G):(x_l))\geq 2$. The graph $G''$ is bipartite with $|V(G'')|=m+n-1$ and $m-1,n\geq 2$. Moreover $G''$ is not complete since $\{x_i,y_j\}\notin E(G'')$. If $G''$ is connected, then we apply the induction hypothesis and obtain that $\depth K[V(G'')]/I(G'')\geq 2$. Otherwise, $G''$ is disconnected and we have noticed that for such a graph $\depth K[V(G'')]/I(G'')\geq 2$. Therefore, we obtain that $\depth S/(I(G),x_l)\geq 2$. Applying now the Depth Lemma to the short exact sequence yields $\depth S/I(G)\geq 2$, as desired.
\end{proof}

\begin{Theorem}
\label{constantdepth}
Let $G$ be a graph without isolated vertices. Then the depth function of the edge ideal $I(G)$ of $G$ is constant if and only if the connected components of $G$ are complete bipartite graphs.
\end{Theorem}

\begin{proof}
First we prove the statement when $G$ is connected. Assume that $G$ is a complete bipartite graph with $V(G)=\{x_1,\ldots,x_m\}\cup\{y_1,\ldots,y_n\}$. Then it follows that $I(G)=(x_1,\ldots,x_m)(y_1,\ldots,y_n)$, hence $I(G)$ is a transversal polymatroidal ideal. Therefore, applying \cite[Corollary 4.14]{HRV} we obtain that $\depth S/I(G)^t=1$ for all $t\geq 1$, as desired. Conversely, assume that $\depth S/I(G)^t$ is constant for all $t$. If $G$ is not bipartite, then $G$ has an odd cycle and from \cite[Corollary 3.4]{CMS} we obtain that $\depth S/I(G)^t=0$ for $t\gg 0$. On the other hand, since $I(G)$ is a squarefree monomial ideal we have $\depth S/I(G)\geq 1$, a contradiction to our assumption that $\depth S/I(G)^t$ is constant for all $t$. Therefore, $G$ must be bipartite and by \cite[Theorem 5.9]{SVV} $I(G)$ is normally torsion free and consequently $\mathcal R(I)$ is a normal Cohen--Macaulay ring. Then it follows that $\lim_{t\to \infty}\depth S/I(G)^t=\dim S-\ell(I(G))$. Since $G$ is bipartite and connected  we obtain that $\lim_{t\to \infty}\depth S/I(G)^t=1$, see \cite[Corollary 10.3.18]{HH}. This implies, according to our hypothesis, that $\depth S/I(G)^t=1$ for all $t$. Hence $\depth S/I(G)=1$ which implies, via Proposition~\ref{bipdepth1}, that $G$ is complete, as desired.

Consider now the case that $G$ is disconnected having the connected components $G_1,\ldots,G_k$ with $k\geq 2$. Assume first that $\depth S/{I(G)^t}$ is constant for all $t\geq 1$. For a graph $G$ the analytic spread of its edge ideal $I(G)$ can be computed as
\[
\ell(I(G))=|V(G)|-c,
\]
where $c\leq k$ is the number of connected bipartite components, see for example \cite[Lemma 8.3.2]{Vi}. Then by the Burch inequality,
\[
\lim_{t\to \infty}\depth S/I(G)^t\leq |V(G)|-\ell(I(G))=c.
\]
On the other hand, since $G$ has $k$ connected components we have $$\depth S/I(G)=\depth K[V(G_1)]/I(G_1)+\cdots+\depth K[V(G_k)]/I(G_k)\geq k.$$ Therefore, $\depth S/{I(G)^t}$ is constant for all $t\geq 1$ implies that $$k=c \quad \text{ and } \quad \depth K[V(G_i)]/I(G_i)=1 \text{ for all } i=1,\ldots,k.$$
In conclusion, we see that $G_i$ is bipartite  and $\depth K[V(G_i)]/I(G_i)=1$ for all $i$, which by Proposition~\ref{bipdepth1} implies that $G_i$ is a complete bipartite graph for all $i$, as desired.

Conversely, let $G_1,\ldots,G_k$ be complete bipartite graphs. Then the depth function of $I(G_i)$ is constant for all $i$, and since $$I(G)=I(G_1)S+\cdots+I(G_k)S,$$
so that Theorem~\ref{constsumprod}(ii) yields the desired conclusion.
\end{proof}

Now we turn our attention to the case of matroidal ideals. We recall that a squarefree monomial ideal $I\subset S=K[x_1,\ldots,x_n]$ generated in a single degree is called {\em matroidal} if the following {\em exchange property} holds: for any $u,v\in G(I)$  and all $i$ such that $x_i|u$ and $x_i\nmid v$, there exists an integer $j\neq i$ such that $x_j|v$, $x_j\nmid u$ and $(u/x_i)x_j\in G(I)$.

For the formulation of the next statement we need to introduce some notation and concepts. For a monomial ideal $I$ such that $G(I)=\{u_1,\ldots,u_m\}$ we set by $\Supp(I)=\bigcup_{i=1}^m \Supp(u_i)$ and $\gcd(I)=\gcd(u_1,\ldots,u_m)$. The {\em linear relation graph} $\Gamma_I$ associated to a monomial ideal $I$ (see \cite[Definition 3.1]{HQ}) is the graph whose vertex set $V(\Gamma_I)$ is a subset of $\{x_1\ldots,x_n\}$ and for which $\{x_i,x_j\}\in E(\Gamma_I)$ if and only if there exist $u_k,u_l\in G(I)$ such that $x_iu_k=x_ju_l$. For our further considerations it is important to recall the fact that for a matroidal ideal $I$ one can compute the analytic spread as $\ell(I)=r-s+1$ (\cite[Lemma 4.2]{HQ}), where $r=|V(\Gamma_I)|$ and $s$ is the number of connected components of $\Gamma_I$.

\begin{Proposition}
\label{criterionmatroidconstdepth}
Let $I\subset S$ be a matroidal ideal generated in degree $d$, and denote as before by $s$ the number of connected components of $\Gamma_I$. Then  $s\leq d$. If in addition $\Supp(I)=\{x_1,\ldots,x_n\}$ and $\gcd(I)=1$, then $V(\Gamma_I)=\{x_1,\ldots,x_n\}$ and $s=d$ if and only if the depth function of $I$ is constant.
\end{Proposition}
\begin{proof}
It is well known that if $I$ is matroidal then $\mathcal{R}(I)$ is normal and hence Cohen--Macaulay (see \cite[Proposition 3.11]{Vi1}). Therefore we have $\lim_{t\to \infty}\depth S/I^t=n-\ell(I)$. Since $\depth S/I\geq \lim_{t\to \infty}\depth S/I^t$ and $\depth S/I=d-1$ (see \cite[Corollary 2.6]{C}), we obtain the inequality $\ell(I)\geq n-d+1$. On the other hand, as observed above, we have $\ell(I)=r-s+1$, where $r=|V(\Gamma_I)|$. This then implies that $r-s+1\geq n-d+1$, or equivalently $d-s\geq n-r$. The conclusion $s\leq d$ follows now since we always have $r\leq n$.

Our additional assumptions $\Supp(I)=\{x_1,\ldots,x_n\}$ and $\gcd(I)=1$ imply that $r=n$. Indeed, if $\Supp(I)=\{x_1,\ldots,x_n\}$ and $\gcd(I)=1$ then for every $i\in [n]$ there exist  $u,v\in G(I)$ such that $x_i$ divides $u$ and does not divide $v$. It follows then from the definition of a matroidal ideal that there exists $j\in [n]$ with $j\neq i$ such that $x_j$ divides $v$ and does not divide $u$ and with the property that $(u/x_i)x_j\in G(I)$. This implies, according to the definition of the linear relation graph, that $\{x_i,x_j\}\in E(\Gamma_I)$. Therefore we have $V(\Gamma_I)=\{x_1,\ldots,x_n\}$ and $r=n$.

Finally, for proving the equivalence stated in the proposition let us notice  that, since $\mathcal{R}(I)$ is Cohen--Macaulay, Corollary~\ref{constlim} together with the first part of our proposition imply that the depth function of $I$ is constant if and only if $n-d=r-s$, that is, if and only if $d=s$, as desired.
\end{proof}

\begin{Lemma}
\label{linrelgraph}
Let $I\subset S$ be a matroidal ideal generated in degree $d$ such that $\Supp(I)=\{x_1,\ldots,x_n\}$ and $\gcd(I)=1$. Then
\begin{equation}\label{conn}
I\subset P_1\cap\cdots\cap P_s,
\end{equation}
where $P_1,\ldots,P_s$ are the monomial prime ideals generated by the sets of vertices of the connected components $\Gamma_1,\ldots,\Gamma_s$ of $\Gamma_I$.
\end{Lemma}

\begin{proof}
It follows from Proposition~\ref{criterionmatroidconstdepth} that $\Gamma_I$ has $s\leq d$ connected components $\Gamma_1,\ldots,\Gamma_s$ and $V(\Gamma_I)=\{x_1,\ldots,x_n\}$. In order to prove (\ref{conn}) we may restrict ourselves  to the case $s\geq 2$. Indeed, if $s=1$ then $P_1=\mm$ and we obviously have $I\subset \mm$. Hence let $s\geq 2$ and assume that (\ref{conn}) does not hold. Then there exists $i$ with $1\leq i\leq s$ such that $I\not\subset P_i$. For the simplicity of notation we assume that $i=1$ and $P_1=(x_1,\ldots,x_k)$, where $k<n$ since $s\geq 2$. Let us observe that $I\not\subset P_1$ implies the existence of a monomial $u\in G(I)$ such that $\Supp(u)\cap \{x_1,\ldots,x_k\}=\emptyset$. On the other hand $\Supp(I)=\{x_1,\ldots,x_n\}$ implies that there exists a monomial $v\in G(I)$ such that $x_1|v$. Therefore, by applying the exchange property between $v$ and $u$ there exists an integer $j>k$ such that $x_j|u$, $x_j\nmid v$ and $w=(v/x_1)x_j\in G(I)$. Hence we obtain from $x_1w=x_jv$ that $\{x_1,x_j\}\in E(\Gamma_I)$, a contradiction since $x_j\notin V(\Gamma_1)$. Consequently we have $I\subset P_1\cap\cdots\cap P_s$.
\end{proof}

\begin{Theorem} \label{constmatroid}
Let $I\subset S$ be a matroidal ideal generated in degree $d$ such that $\gcd(I)=1$ and $\Supp(I)=\{x_1,\ldots,x_n\}$. Then the depth function of $I$ is constant if and only if $I=P_1\cdots P_d$, where $P_1,\ldots,P_d$ are monomial prime ideals in pairwise disjoint sets of variables.
\end{Theorem}

\begin{proof}
Assume first that $I=P_1\cdots P_d$, where $P_1,\ldots,P_d$ are monomial prime ideals in pairwise disjoint sets of variables. It follows from \cite[Theorem 4.12]{HRV} and \cite[Corollary 4.14]{HRV} that $\depth S/I=\lim_{t\to \infty}\depth S/I^t=d-1$. Therefore, due to the fact that the Rees ring of a matroidal ideal is Cohen--Macaulay,  the depth function of $I$ is constant.

For the converse, let us notice first that since the depth function of $I$ is constant,  Proposition~\ref{criterionmatroidconstdepth} implies that $\Gamma_I$ has $d$ connected components $\Gamma_1,\ldots,\Gamma_d$. Moreover, if for all $i$ we denote by $P_i$ the monomial prime ideal  generated by the set of vertices of $\Gamma_i$, then by Lemma~\ref{linrelgraph}
\[
I\subset P_1\cap\cdots\cap P_d=P_1\cdots P_d.
\]

We will prove that $I=P_1\cdots P_d$. For this we first relabel the set of variables $\{x_1,\ldots,x_n\}$ suitably to indicate to which $P_i$ they belong,  that is, we write
\[
P_i=(x_{i1},\ldots,x_{ik_i}) \quad \text{for}\quad i=1,\ldots d.
\]
Then $\{x_1,\ldots,x_n\}=\Union_{i=1}^d \{x_{i1},\ldots,x_{ik_i}\}$, where $k_1+\cdots+k_d=n$ and $k_1,\ldots,k_d\geq 2$ since $\Supp(I)=\{x_1,\ldots,x_n\}$ and $\gcd(I)=1$. We also write $I=x_{11}J_1+\cdots+x_{1k_1}J_{k_1}$, where $J_l=I:(x_{1l})\subset P_2\cdots P_d$  for all $l$, since $I\subset P_1\cdots P_d$. Moreover, we have for all $l$ that $J_l$ is a matroidal ideal generated in degree $d-1$, see for example \cite[Theorem 1.1]{BH1}. In this setting, to prove that $I=P_1\cdots P_d$ is equivalent to showing that $J_1=\cdots=J_{k_1}=P_2\cdots P_d$. We prove this latter statement by induction on $d$.

Assume first that $d=2$. We show that $J_l=P_2$ for all $l$ with $1\leq l\leq k_1$. Since $\Supp(I)=\{x_{11},\ldots,x_{1k_1}\}\cup\{x_{21},\ldots,x_{2k_2}\}$, there exists a monomial $u\in G(I)$ such that $x_{1l}|u$. Without loss of generality, we may assume that $u=x_{1l}x_{21}$. Let $j$ be an arbitrary integer such that $2\leq j\leq k_2$ and $v=x_{1i_j}x_{2j}\in G(I)$ for some $i_j\leq k_1$. If $i_j=l$ then $x_{2j}\in J_l$. Otherwise, applying the exchange property between $v$ and $u$ we obtain $(v/x_{1i_j})x_{1l}\in G(I)$ hence $x_{2j}\in J_l$. Consequently, since $j$ was chosen arbitrarily we have $J_l=(x_{21},\ldots,x_{2k_2})$.

For the induction step assume that $I$ is generated in degree $d$. Since $I\subset P_1\cdots P_d$ it follows that all monomials of $I$ are of the form $x_{1i_1}\cdots x_{di_d}$ for some integers $i_1,\ldots,i_d$ such that $i_t\leq k_t$ for all $t$. Moreover we know that  $J_1,\ldots,J_{k_1}\subset P_2\cdots P_d$ are matroidal ideals generated in degree $d-1$. The desired conclusion follows at once if we show that $\gcd(J_l)=1$ and $\Supp(J_l)=\bigcup_{j=2}^d \{x_{j1},\ldots,x_{jk_j}\}$ for all $l$. Indeed, if this is the case then our induction hypothesis implies  that $J_l=P_2\cdots P_d$ for all $l$,  and consequently $I=P_1\cdots P_d$.

It is enough to prove that $\gcd(J_l)=1$ and $\Supp(J_l)=\bigcup_{j=2}^d \{x_{j1},\ldots,x_{jk_j}\}$ only for $l=1$, the other cases being analogous to this one. Since $x_{11}\in\Supp(I)$ we may assume that the monomial $u'=x_{11}x_{21}\cdots x_{d1}$ belongs to $G(I)$, or equivalently $x_{21}\cdots x_{d1}\in J_1$. First we show that $\Supp(J_1)=\bigcup_{j=2}^d \{x_{j1},\ldots,x_{jk_j}\}$. We choose a variable $x_{il}$ for some integers $i,l$ with $2\leq i\leq d$ and $1\leq l\leq k_i$. Since $x_{il}\in\Supp(I)$ it follows that there exist integers $j_1,\ldots,j_{i-1},j_{i+1},\ldots,j_d$ such that the monomial
\[
v'=x_{1j_1}\cdots x_{i-1 j_{i-1}}x_{il}x_{i+1 j_{i+1}}\cdots x_{dj_d}\in G(I).
\]

If $j_1=1$, then by the definition of $J_1$ we have $x_{il}\in\Supp(J_1)$. Otherwise applying the exchange property for $u'$ and $v'$ we obtain that the monomial $(v'/x_{1j_1})x_{11}\in G(I)$. Therefore,   $x_{il}\in\Supp(J_1)$, and we are done. Secondly, we prove that $\gcd(J_1)=1$. Indeed, assume by contradiction that this is not the case. Then there exists a variable, say $x_{21}$, which divides all $w\in G(J_1)$. Let
\[
w_1=x_{21}x_{3j_3}\cdots x_{dj_d}\in G(J_1)
\]
be such a monomial. Since $x_{22}\in\Supp(I)$,  there exists a monomial $w_2\in G(I)$ of the form $w_2=x_{1i_1}x_{22}x_{3i_3}\cdots x_{di_d}$. Our assumption that $x_{21}|\gcd(J_1)$ implies that $i_1\neq 1$. Applying now the exchange property between the monomials $w_2,x_{11}w_1\in G(I)$ with respect to the variable $x_{1i_1}$, we obtain $(w_2/x_{1i_1})x_{11}\in G(I)$. Hence $x_{22}x_{3i_3}\cdots x_{di_d}\in J_1$, a contradiction to our assumption. Therefore $\gcd(J_1)=1$, and we are done.
\end{proof}

A consequence of the previous theorem is the following nice fact. Let $I$ be a matroidal ideal generated in degree $d$ such that $\gcd(I)=1$, $\Supp(I)=[n]$ and $G(I)\subset G(P_1\cdots P_d)$, where $P_1,\ldots,P_d$ are monomial prime ideals generated in pairwise disjoint sets of variables with the property that $P_1+\cdots+P_d=\mm$. Then $I=P_1\cdots P_d$, and in particular $I$ is a transversal matroidal ideal.

\medskip

Finally we consider the facet ideal of a forest. Let $\Delta$ be a simplicial complex of dimension $d$ on the set $[n]$. We denote by $\mathcal{F}(\Delta)$ the set of facets of $\Delta$ and by $I(\Delta)\subset S=K[x_1,\ldots,x_n]$ the {\em facet ideal} of $\Delta$, whose generators are the monomials $x_F=\prod_{i\in F}x_i$ for all $F\in\mathcal{F}(\Delta)$. The simplicial complex $\Delta$ is called {\em pure} if all the facets have the same dimension. A pure simplicial complex of dimension $d$ is said to be {\em connected in codimension one} if for any two facets $F,G\in\mathcal{F}(\Delta)$ there exist facets $F=F_1,F_2,\ldots,F_r=G$ such that $\dim (F_i\cap F_{i+1})=d-1$ for all $i$.

A facet $F\in\mathcal{F}(\Delta)$ is called a {\em leaf} if either $\mathcal{F}(\Delta)=\{F\}$ or there exists a facet $G\neq F$ such that $F\cap H\subseteq F\cap G$ for all $H\in\mathcal{F}(\Delta)$. The facet $G$ with this property is called a {\em branch} of $F$. A vertex $i$ of $\Delta$ is called a {\em free vertex} of $\Delta$ if $i$ belongs to exactly one facet. Observe that every leaf has at least one free vertex. Faridi \cite{Fa} calls a simplicial complex $\Delta$ a {\em simplicial forest} if each simplicial complex $\Gamma$ with $\mathcal{F}(\Gamma)\subset\mathcal{F}(\Delta)$ has a leaf. A connected simplicial forest is called a {\em simplicial tree}.
\texttt{}

\begin{Theorem}
\label{consttree}
Let $\Delta$ be a pure simplicial forest with the property that each connected component $\Delta_1,\ldots,\Delta_k$ is connected in codimension one. Then the following conditions are equivalent:
\begin{enumerate}
\item[(a)] The depth function of $I(\Delta)$ is constant;
\item[(b)] For all $i$, $I(\Delta_i)$ is a product of monomial prime ideals such that at most one of the factors is not principal;
\item[(c)] $I(\Delta)$ belongs to the class $\mathcal{C}$ defined in the introduction.
\end{enumerate}
\end{Theorem}
\begin{proof}
The implication (b)$\Rightarrow$(c) is obvious and (c)$\Rightarrow$(a) follows immediately from Example~\ref{good}(ii). It remains to prove the implication (a)$\Rightarrow$(b).  We recall that for a simplicial forest the Rees algebra of its facet ideal is Cohen--Macaulay, see for example \cite[Proposition 10.3.21]{HH}. Now since $I(\Delta)=I(\Delta_1)S+\cdots+I(\Delta_k)S$ and the Rees algebras $\mathcal{R}(I(\Delta))$, $\mathcal{R}(I(\Delta_i))$ are Cohen--Macaulay for all $i$ we obtain by Theorem~\ref{constsumprod} that $I(\Delta)$ has constant depth if and only if each $I(\Delta_i)$ has constant depth. Therefore, we may assume that $\Delta$ is connected. Let $\mathcal{F}(\Delta)=\{F_1,\ldots,F_m\}$. Since $\Delta$ is a pure tree, $\ell(I(\Delta))=m$ and $\lim_{t\to \infty}\depth S/I(\Delta)^t=n-m$, see \cite[Corollary 10.3.22]{HH}. On the other hand, $\Delta$ being connected in codimension one implies that $m=n-d$, where $d=\dim \Delta$. Therefore,  $\lim_{t\to \infty}\depth S/I(\Delta)^t=d$.

Note, that if $m=1$,  then $I(\Delta)$ is a principal ideal, hence has a constant depth function, and $I(\Delta)$ is the product of principal monomial prime ideals. Therefore we can restrict ourselves to the case $m\geq 2$. Assume now that $I(\Delta)$ is not of the form described in (b), that is $I(\Delta)\neq P\cdot u$, where $P$ is a monomial prime ideal and $u$ is a squarefree monomial with $u\not\in P$. We will prove then that $\depth S/I(\Delta)>d$, which combined with $\lim_{t\to \infty}\depth S/I(\Delta)^t=d$ yields a contradiction to our hypothesis that the depth function of $I(\Delta)$ is constant.

First we show that there exist two leaves, say $F_1$ and $F_m$, such that $\dim F_1\cap F_m < d-1$. We may assume that $F_1=\{x_{i_1},\ldots,x_{i_{d+1}}\}$ is a leaf of $\Delta$. (Here and in the following we identify a vertex $i$ with the variable $x_i$.) Let $G_1,\ldots,G_r$ be the branches of $F_1$, say, $G_k\sect F_1=\{x_{i_1},\ldots,x_{i_d}\}$ for all $k$, and $G_k\setminus F_1=\{x_{j_k}\}$ for  $k=1,\ldots,r$.
In the case that  all $G_k$ are leaves we observe that $\mathcal{F}(\Delta)=\{F_1,G_1,\ldots,G_r\}$. Indeed, suppose that this is not the case, then there exists a facet $F$ of $\Delta$ which is not a branch of $F_1$ and which intersects some $G_k$ in codimension one. Since $G_k$ is a leaf and both, $F$ and $F_1$,  intersect $G_k$ in codimension one it follows that $F\sect G_k = F_1\sect G_k$, which implies that $F$ is a branch of $F_1$, a contradiction.
Now since $\mathcal{F}(\Delta)=\{F_1,G_1,\ldots,G_r\}$ it follows that $I=Pu$ where $P=(x_{i_{d+1}},x_{j_1},\ldots,x_{j_r})$ and $u=\prod_{k=1}^rx_{i_k}$, a contradiction. Therefore, we may assume that for some integer $1\leq s<r$,  $G_1,\ldots, G_s$ are non-leaves, while $G_{s+1},\ldots,G_r$ are leaves.  Removing the leaves $G_{s+1},\ldots,G_r$ from $\Delta$  we obtain a subcomplex $\Gamma$ which is again a tree and for which each leaf of $\Gamma$ is also  a leaf of $\Delta$. Since $\Gamma$ is tree it  has at least two leaves, one of them  being $F_1$ and  another leaf, say  $F_m$. Since $F_m\neq G_k$ for $k=1,\ldots,s$, it follows that $F_m$ is not a branch of $F_1$. Thus $\dim F_m\sect F_1<d-1$.

Assume now for simplicity that $F_1=\{x_1,\ldots,x_{d+1}\}$ and that $x_{d+1}$ is the  free vertex of $F_1$. In addition, we may also assume that $F_1\cap F_m\subseteq\{x_1,\ldots,x_{d-1}\}$,  and that there exist integers $i,j$ with $i\neq j$ and $i,j>d+1$ such that $x_i,x_j\in F_m$ and that $x_i$ is a free vertex of $F_m$. Consider the $K$-algebra
\[
A=K[x_{d+1},x_{d+2},\ldots, x_{j-1},x_{j+1}\ldots,x_n]/(x_{d+1}^{d+1},u_2,\ldots,u_m),
\]
where $u_k$ is the monomial obtained from $x_{F_k}$ by replacing each of the variables $x_1,\ldots,x_d$ with $x_{d+1}$, and replacing $x_j$ with $x_i$. Then it follows that $S/I(\Delta)$ is the polarization of the $K$-algebra $A$ and that $(S/I(\Delta)/({\bf z})(S/I(\Delta))\iso A$, where ${\bf z}=x_{d+1}-x_1,\ldots,x_{d+1}-x_d,x_i-x_j$,  see for example \cite[Proposition 1.6.2]{HH}. This shows that $\depth S/I(\Delta)>d$, since ${\bf z}$ is a regular sequence of length $d+1$ on $S/I(\Delta)$.
\end{proof}

The following examples show that both conditions, namely $\Delta$ being pure and $\Delta$ being connected in codimension one, are required for the implication (a)$\Rightarrow$(b) of Theorem~\ref{consttree}.

\begin{Example}\label{extrees}
{\em  (i)  Let $\Delta$ be the non--pure simplicial tree on the  vertex set $[5]$, whose facets are $\mathcal{F}(\Delta)=\{\{1,2,3\},\{1,5\},\{3,4\}\}$. Then  the  facet ideal $I(\Delta)=(x_1x_2x_3,x_1x_5,x_3x_4)$ is obviously not of the form given in Theorem~\ref{consttree}(b). By using CoCoA \cite{Co} one can easily see that $\depth S/I(\Delta)=2=5-3= n-\ell(I)$.  Since $\Delta$ is a simplicial tree, $\mathcal{R}(I(\Delta))$ is Cohen--Macaulay (see \cite[Corollary 4]{Fa}). Thus we may apply Corollary~\ref{constlim} and  obtain that the depth function of $I(\Delta)$ is constant.

(ii) Let $\Gamma$ be the pure simplicial tree on the vertex set $[7]$, whose facets are $\mathcal{F}(\Gamma)=\{\{1,2,3\},\{3,4,5\},\{5,6,7\}\}$. Observe that $\Gamma$ is not connected in codimension one and $I(\Gamma)$ is not of the form given in Theorem~\ref{consttree}(b). Nevertheless,  $\Gamma$ is a pure simplicial tree, so that  $\mathcal{R}(I(\Gamma))$ is Cohen--Macaulay. Thus we may again apply  Corollary~\ref{constlim}. Checking with CoCoA the relevant data  we see that the depth function of $I(\Gamma)$ is constant.
}
\end{Example}

\end{document}